\documentclass[12pt]{amsart}
\usepackage{fullpage}
\usepackage{amsmath,amsthm,amssymb,latexsym}
\usepackage{exscale}
\usepackage{relsize}
\usepackage{graphicx}
\usepackage{tikz}
\usepackage{multicol}
\usepackage{color,soul}

\theoremstyle{plain}
\newtheorem*{theorem-non}{Theorem}
\newtheorem*{corollary-non}{Corollary}
\newtheorem{theorem}{Theorem}[section]

\newtheorem{lemma}[theorem]{Lemma}

\newtheorem{proposition}[theorem]{Proposition}

\newtheorem{algorithm}[theorem]{Algorithm}

\theoremstyle{definition}
\newtheorem{definition}[theorem]{Definition}

\newtheorem{example}[theorem]{Example}

\theoremstyle{remark}
\newtheorem{remark}[theorem]{Remark}

\numberwithin{figure}{section}

\newcommand{\Z}{\mathbb{Z}}                
\newcommand{\mx}[1]{\mathbf{#1}}           

\renewcommand{\emptyset}{\varnothing}

\begin{document}

\title{Splines Over Integer Quotient Rings}
\author{McCleary Philbin, Lindsay Swift, Alison Tammaro, \\
Danielle Williams }

\maketitle

\begin{abstract}
Given a graph with edges labeled by elements in $\Z/m\Z$, a generalized spline is a labeling of each vertex by an integer $\mod m$ such that the labels of adjacent vertices agree modulo the label associated to the edge connecting them. These generalize the classical splines that arise in analysis as well as in a construction of equivariant cohomology often referred to as GKM-theory.  
We give an algorithm to produce minimum generating sets for the $\Z$-module of splines on connected graphs over $\Z/m \Z$. As an application, we give a quick heuristic to determine the minimum number of generators of the module of splines over $\Z/m \Z$. We also completely determine the ring of splines over $\Z/p^k\Z$ by providing explicit multiplication tables with respect to the elements of our minimum generating set. Our final result extends some of these results to splines over $\Z$.
\end{abstract}


\section{Introduction}

We consider an extension of classical splines that was introduced by Gilbert, Tymoczko, and Viel \cite{GilbertPolsterTymoczko}.  Classical splines can be thought of as piecewise polynomials over a polyhedral complex that agree up to a specified degree of smoothness at the intersection of faces.  By contrast, generalized splines are defined as follows.  Fix a ring $R$, a graph $G=(V,E)$, and a map $\alpha: E \to \{\text{ideals of } R\}$ called an edge-labeling function. A spline on the edge-labeled graph $(G,\alpha)$ is a vertex-labeling $\mx{f} \in R^{|V|}$ so that for each edge $uv$ the difference $\mx{f}_u-\mx{f}_v$ is in the ideal $\alpha(uv)$.  When $R$ is a polynomial ring and $G$ is the dual graph to a polyhedral complex, we recover the classical case; this dual setting is how splines arise in algebraic topology, especially when computing equivariant cohomology via what is called GKM-theory.

Two central problems about splines in the classical setting are: 1) compute the dimension and 2) find an explicit basis of the complex vector space of splines (see the book of Lai-Schumaker for a survey \cite{Lai-Schumaker}).  Gilbert-Tymoczko-Viel showed that when the ring $R$ is a domain, the dimension is simply the number of vertices in the graph \cite[5.2]{GilbertPolsterTymoczko}. (This corresponds to the case of $\mathcal{C}^\infty$-splines, for which the coefficient ring $R$ is a polynomial ring.)  Bowden-Tymoczko showed that this is not true when the ring $R$ has zero-divisors, by contrast, and in fact gave examples when the spline space achieves each dimension between $2$ and $|V|$ \cite[4.2]{BT}. (This corresponds to the case of splines with bounded degree $d$, for which the coefficient ring $R$ can be taken to be a quotient of a polynomial ring by the ideal generated by all monomials of degree $d+1$.)

In this paper we complete Bowden-Tymoczko's analysis of splines over the ring $\Z/m\Z$.  They were only able to describe splines over cycles; we give an explicit basis for splines on arbitrary graphs over $\Z/m\Z$ as well as an easy heuristic to compute the dimension of splines for any graph over $\Z/m\Z$.  

In order to describe our results in more detail, we need notation.  The ring $\Z/m\Z$ is a PID so we can think of the edge-labeling function $\alpha$ as taking images in $\Z/m\Z$.  Indeed, we often treat the edge-labeling function as a map $\alpha: E \rightarrow \Z$ and implicitly identify the images with the ideals $\Z \alpha(uv)+m\Z/m\Z$.  Given a graph $G$ and an edge-labeling function $\alpha: E \rightarrow \Z$ we write $\mathbb{B}_m$ for a minimum generating set of $(G,\alpha)$ with respect to the ring $\Z/m\Z$.

Theorem \ref{rank theorem} computes the rank of the space of splines over $\Z/m\Z$ for arbitrary $m$, which we paraphrase below.

\begin{theorem-non} 
Let $G$ be a graph and $\alpha: E \rightarrow \Z$ an edge-labeling function.  The number of elements in $\mathbb{B}_{p^e}$ is the same as the number of connected components of the graph obtained from $G$ by removing all edges except those for which $\alpha(uv) \in p^e \Z$.  Given a positive integer $m$ let $m=p_1^{e_1}p_2^{e_2}\cdots p_t^{e_t}$ be the primary decomposition of $m$.  The number of elements in $\mathbb{B}_m$ is the maximum of the number of elements in $\mathbb{B}_{p_i^{e_i}}$ over $i=1,\ldots,t$.
\end{theorem-non}

The components in the previous theorem partition the vertices so that two vertices are in the same path if they are connected by edges labeled $0$ when the edge-labeling function $\alpha$ is reduced mod $m$.  We refer to the parts of this partition as {\em zero-connected components (mod $m$),} though we omit the modulus if it is clear in context.

\begin{example}
We find the size of a minimum generating set over $\Z/10\Z$ for splines on the graph in Figure \ref{graph mod 10} by examining the number of zero-connected components of the graph considered mod $2$ in Figure \ref{graph mod 2} and mod $5$ in Figure \ref{graph mod 5}. There are four zero-connected components in Figure \ref{graph mod 5} and two in Figure \ref{graph mod 2}, so there are four and two elements in a minimum generating set for splines on this graph over $\Z/5\Z$ and $\Z/2\Z$ respectively. Therefore there are four elements in a minimum generating set for splines on the graph in Figure \ref{graph mod 10} over $\Z/10\Z$. 

\begin{figure}[h!]

\begin{minipage}[l]{.32\linewidth}
\begin{center}
\begin{tikzpicture}[shorten >=1pt,auto,node distance=1.5cm,
  thick,main node/.style={rectangle,draw,font=\sffamily\bfseries}]
  
 \node[main node] (1) {$v_1$};
 \node[main node] (2) [above right of = 1] {$v_2$};
 \node[main node] (3) [above of = 2] {$v_3$};
 \node[main node] (5) [above left of =1] {$v_5$};
  \node[main node] (4) [above of =5] {$v_4$};

  \path[every node/.style={color=black, font=\sffamily\small}]
    
    (2) edge node {\color{red}$2$} (1)
    	edge node {\color{red}$2$} (5)
    (1) edge node {\color{red}$1$} (5)
    (4) edge node {\color{red}$2$} (3)
    	edge node {\color{red}$1$} (2)
    (3) edge node {\color{red}$1$} (2)
    (5) edge node {\color{red}$5$} (4);

\end{tikzpicture}
\caption{ Graph over $\Z/10\Z$}\label{graph mod 10}
\end{center}
\end{minipage}
\begin{minipage}[l]{0.32\linewidth}
\begin{center}
\begin{tikzpicture}[shorten >=1pt,auto,node distance=1.5cm,
  thick,main node/.style={rectangle,draw,font=\sffamily\bfseries}]
  
 \node[main node] (1) {$v_1$};
 \node[main node] (2) [above right of = 1] {$v_2$};
 \node[main node] (3) [above of = 2] {$v_3$};
 \node[main node] (5) [above left of =1] {$v_5$};
  \node[main node] (4) [above of =5] {$v_4$};

  \path[every node/.style={color=black, font=\sffamily\small}]
    
    (2) edge node {\color{red}$0$} (1)
    	edge node {\color{red}$0$} (5)
    (1) edge node {\color{red}$1$} (5)
    (4) edge node {\color{red}$0$} (3)
    	edge node {\color{red}$1$} (2)
    (3) edge node {\color{red}$1$} (2)
    (5) edge node {\color{red}$1$} (4);

\end{tikzpicture}
\caption{ Graph over $\Z/2\Z$}\label{graph mod 2}
\end{center}
\end{minipage}
\begin{minipage}[l]{0.32\linewidth}
\begin{center}
\begin{tikzpicture}[shorten >=1pt,auto,node distance=1.5cm,
  thick,main node/.style={rectangle,draw,font=\sffamily\bfseries}]
  
 \node[main node] (1) {$v_1$};
 \node[main node] (2) [above right of = 1] {$v_2$};
 \node[main node] (3) [above of = 2] {$v_3$};
 \node[main node] (5) [above left of =1] {$v_5$};
  \node[main node] (4) [above of =5] {$v_4$};

  \path[every node/.style={color=black, font=\sffamily\small}]
    
    (2) edge node {\color{red}$1$} (1)
    	edge node {\color{red}$1$} (5)
    (1) edge node {\color{red}$1$} (5)
    (4) edge node {\color{red}$1$} (3)
    	edge node {\color{red}$1$} (2)
    (3) edge node {\color{red}$1$} (2)
    (5) edge node {\color{red}$0$} (4);

\end{tikzpicture}
\caption{Graph over $\Z/5\Z$}\label{graph mod 5}
\end{center}
\end{minipage}
\end{figure}
\end{example}

In fact, over $\Z/p^e\Z$ we can construct a basis explicitly, essentially by using the zero-connected components as the support of certain splines.  More formally we have the following.

\begin{theorem-non}
Fix an edge-labeled graph $(G,\alpha)$ and a prime power $p^e$.  There is a minimum generating set $\mathbb{B}_{p^e}$ for the splines on $(G,\alpha)$ over $\Z/p^e\Z$ that satisfies the following properties:
\begin{itemize}
\item The support of each spline $\mx{b}^i \in \mathbb{B}_{p^e}$ is a zero-connected component mod $p^k$ for some $k$ with $0 \leq k \leq e$.
\item If $k$ is the minimum integer such that the support of $\mx{b}^i \in \mathbb{B}_{p^e}$ is a zero-connected component mod $p^k$ then $\mx{b}^i_u = p^k$ for all $u$ in that zero-connected component and $\mx{b}^i_u = 0$ for all other $u$.
\end{itemize}
\end{theorem-non}

Algorithm \ref{mod p^k algorithm} constructs this minimum generating set explicitly. Moreover Lemma \ref{support lemma} shows that the supports of each pair of splines $\mx{b}^i$ and $\mx{b}^j$ in $\mathbb{B}_{p^e}$ are either disjoint or nested, via analogous properties for zero-connected components modulo different prime powers.

Using appropriate pullback homomorphisms, we can generalize these results and give a heuristic to find minimum generating sets $\mathbb{B}$ for splines over $\Z/m\Z$.  Results of Bowden and Tymoczko \cite{BT} and of Bowden, Hagen, King, and Reinders \cite{BowdenHKR} also play key roles in that proof.  Theorem \ref{integer theorem} then describes how to choose an appropriate modulus $m$ in order to extend these results to splines over $\Z$.

Finally, we use the explicit description of the minimum generating set mod $p^e$ to explicitly compute the structure constants in the ring of splines.  This is a central question about splines in the context of equivariant cohomology.  The following summarizes Theorem \ref{prime power mult}.

\begin{theorem-non}
Suppose that $\mathbb{B}_{p^e}$ is the minimum generating set constructed in Algorithm \ref{mod p^k algorithm}.  Let $\mx{b^{(i)}}$ be an element of $\mathbb{B}_{p^e}$ with non-zero entries $p^s$ and $\mx{b^{(j)}}$ be an element of $\mathbb{B}_{p^e}$ with non-zero entries $p^r$ where $s\leq r$.  Then we have
\[
\mx{b^{(i)}}\mx{b^{(j)}} =
\begin{cases}
p^s\mx{b^{(j)}}&  \text{if their supports are nested and  } s+r < k\\ 
\mx{0} & \text{otherwise} \\
\end{cases}
\] 
\end{theorem-non}

While we use the rings $R=\Z$, $R=\Z/m\Z$ and $R= \Z/p^k\Z$ to compute splines, the results of this paper hold for splines over any principal ideal domain. In general, to translate the results and proofs one need only change $\Z$ to a general PID $R$ and $\Z/m\Z$ to $R/I$ where $I$ is an ideal of $R$, and $\Z/p^k\Z$ to $R/P^k$ where $P$ is a prime ideal of $R$. 

We wonder whether the techniques in this paper be adapted to polynomial splines of bounded degree, even perhaps with particular restrictions on edge-labels, such as arises in applications.

\section{Background and Notation}\label{background}

In this section we present an overview of important definitions following the notation of \cite{BT}. Our convention throughout this document is that $p$ is prime, $m$ is an integer, and $|V|$ is the number of vertices of a graph.
 
\begin{definition}
Let $G = (V,E)$ be a finite connected graph. Let $R$ be a commutative ring with identity. Let $\alpha: E \to \{\text{ideals in  } R\}$ be a function that labels the edges of $G$ with ideals in $R$. A spline on $G$ is a vertex-labeling $\mx{f} \in R^{|V|}$ such that for each edge $v_iv_j \in E$ we have

\[\mx{f}_{v_i} - \mx{f}_{v_j} \in  \alpha(v_iv_j)  \]

\noindent
where $\mx{f}_{v_i}$ denotes the label of vertex $v_i$. The collection of splines over the graph $G$ with edge-labeling $\alpha$ is denoted $R_{G,\alpha}$. 
\end{definition}

\begin{remark}\label{edge convention}
In this paper the base ring is either $R = \Z/m\Z$ or $R = \Z$. Since both $\Z/m\Z$ and $\Z$ are principal ideal domains, we typically denote the edge label $\langle g \rangle$ by a generator $g \in \Z/m\Z$. In fact if $g \in \Z$ is the minimal positive element of its coset $g + m\Z$ we typically denote the edge label by the integer $g$.  
\end{remark}

\begin{remark}
Suppose $R = \Z/m\Z$ or $R = \Z$.  Then a spline on an edge-labeled graph $(G,\alpha)$ is an element $\mx{f}$ in $(\Z/m\Z)^{|V|}$ or $\Z^{|V|}$ respectively such that for each edge $v_iv_j \in E$ we have

\[\mx{f}_{v_i} - \mx{f}_{v_j} \equiv 0 \mod \alpha(v_iv_j)\]
\end{remark}

\begin{example}
Figure \ref{spline example} depicts an edge-labeled graph with the corresponding relations on the vertex labels.
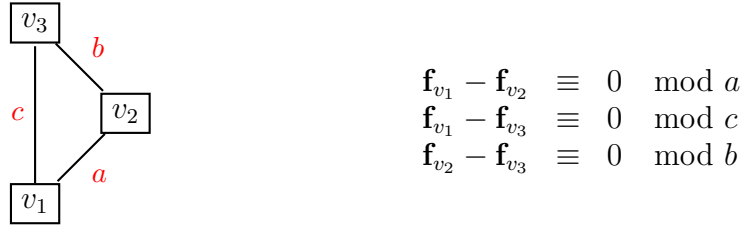
\begin{figure}[h!]

\begin{minipage}[c]{0.3\linewidth}

\begin{tikzpicture}[shorten >=1pt,auto,node distance=1.7cm,
  thick,main node/.style={rectangle,draw,font=\sffamily\bfseries}]
  
 \node[main node] (1) {$v_1$};
 \node[main node] (2) [above right of = 1] {$v_2$};
 \node[main node] (3) [above left of = 2] {$v_3$};

  \path[every node/.style={color=black, font=\sffamily\small}]
    
    (2) edge node {\color{red}$a$} (1)
    (1) edge node {\color{red}$c$} (3)
    (3) edge node {\color{red}$b$} (2);

\end{tikzpicture}
\end{minipage}
\hspace{2mm}
\begin{minipage}[c]{0.2\linewidth}

\[\begin{array}{l l l}
\mx{f}_{v_1}-\mx{f}_{v_2} & \equiv & 0 \mod a\\
\mx{f}_{v_1}-\mx{f}_{v_3} & \equiv & 0 \mod c\\
\mx{f}_{v_2}-\mx{f}_{v_3} & \equiv & 0 \mod b\\
\end{array}\]

\end{minipage}
\caption{Edge-Labeled Graph and Relations on Vertices} \label{spline example}
\end{figure}

We represent splines using standard vector notation. In Figure \ref{spline example}, we have $|V|=3$ and

\[\mx{f}^T = \left( \mx{f}_{v_3} \,, \mx{f}_{v_2} \,, \mx{f}_{v_1}\right)^T \]

\end{example}

\begin{remark}\label{zero_edge_same_label_}
When an edge $e=v_iv_j$ is labeled with a unit in the base ring $R$ there is no restriction on the labels of the vertices of that edge since $x -y \in \langle 1 \rangle$ holds for all $x,y \in R$. If an edge $e = v_iv_j$ is labeled zero then for every spline $\mx{f}$ we must have $\mx{f}_{v_i} = \mx{f}_{v_j}$ since $x-y \equiv 0 \mod 0$ implies $x=y$.
\end{remark}

In this paper we are interested in particular types of splines called \emph{flow-up} splines. A collection of flow-up splines generalizes the concept of a triangular basis from linear algebra. To establish a ``flow" we order the vertices. We adopt the terminology ``flow-up" from certain geometric contexts in which the elements are defined by torus flows \cite{JT}.

\begin{definition}
Given a graph $G$ with an ordered set of vertices $\{v_1, v_2, \dots, v_n\}$ a flow-up spline for a vertex $v_i$ is a spline $\mx{f^{(i)}}$ for which $\mx{f^{(i)}}_{v_t} = 0$ whenever $t<i$. 
\end{definition}

In some cases the entries of a flow-up spline have at most one non-zero value. 

\begin{definition}
A constant flow-up spline in $R_{G,\alpha}$ is a flow-up spline $\mx{f^{(i)}}$ with 
$\mx{f^{(i)}}_{v_t}= r \in R$ for each vertex $v_t \in V$ and $t\geq i$.
\end{definition}

In later sections we compare splines on the same graph over different quotient rings using the standard quotient map $\rho: R \to R/ I $. We apply the quotient map to the edge-labeling function $\alpha$ to get the edge-labeling function 
\[\alpha_I: E \to \text{ ideals of }  R/ I \] 
defined by $\alpha_I(uv)=\rho(\alpha(uv))$ and apply the induced quotient map $\rho^*:R_{G,\alpha} \to R_{G,\alpha_I}$ to the ring of splines. 

\begin{example} Figure \ref{mod 6 graph} and Figure \ref{mod 3 graph} show the quotient map $\rho_{3}: \Z/6\Z \to \Z/3\Z$ for a particular graph $G$.

\begin{figure}[h!]
    \centering
\begin{minipage}[c]{0.45\linewidth}

\begin{tikzpicture}[shorten >=1pt,auto,node distance=1.7cm,
  thick,main node/.style={rectangle,draw,font=\sffamily\bfseries}]
  
 \node[main node] (1) {$v_1$};
 \node[main node] (2) [above right of = 1] {$v_2$};
 \node[main node] (3) [above left of = 2] {$v_3$};

  \path[every node/.style={color=black, font=\sffamily\small}]
    
    (2) edge node {\color{red}$2$} (1)
    (1) edge node {\color{red}$2$} (3)
    (3) edge node {\color{red}$3$} (2);

\end{tikzpicture}
\centering
\caption{Graph over $\Z/6\Z$}\label{mod 6 graph}
\end{minipage}
    \quad
\begin{minipage}[c]{0.45\linewidth}

\begin{tikzpicture}[shorten >=1pt,auto,node distance=1.7cm,
  thick,main node/.style={rectangle,draw,font=\sffamily\bfseries}]
  
 \node[main node] (1) {$v_1$};
 \node[main node] (2) [above right of = 1] {$v_2$};
 \node[main node] (3) [above left of = 2] {$v_3$};

  \path[every node/.style={color=black, font=\sffamily\small}]
    
    (2) edge node {\color{red}$1$} (1)
    (1) edge node {\color{red}$1$} (3)
    (3) edge node {\color{red}$0$} (2);

\end{tikzpicture}
\centering
\caption{Graph over $\Z/3\Z$}\label{mod 3 graph}
\end{minipage}
\end{figure}

Note that  $\rho_3  (\langle2\rangle) = \langle1\rangle$ since $\langle 2 \rangle = \langle 1 \rangle$ in $\Z/3\Z$. Therefore we label the corresponding edges $1$ following the convention stated in Remark \ref{edge convention}.

\end{example}

\begin{definition}
We use the terminology \emph{minimum generating set} for a spanning set of splines with the smallest possible number of elements. These sets function like bases except there might be torsion, namely non-zero scalars $c_b$ such that $c_b\mx{b}=\mx{0}$. 
\end{definition}

\section{Minimum generating sets} \label{section algorithm mod m}

In this section we present an algorithm to produce flow-up minimum generating sets for splines with $\Z/p^k\Z$, $\Z/m\Z\,$ and $\Z$ as base rings. We begin with a definition and discussion of zero-connected components, then outline the steps for the algorithm that produces flow-up minimum generating sets for splines over $\Z/p^k\Z$. We then give an overview of some previous work done by Bowden and Tymoczko \cite{BT} and use it to pull back to minimum generating sets over $\Z/m\Z$ from the primary decomposition $\bigoplus_{i=1}^k \Z/p_i^{e_i}\Z$. Finally we extend these results to give an algorithm to produce minimum generating sets for splines over the integers.

\subsection{An algorithm to produce minimum generating sets modulo $p^k$}\label{mod p^k sect}
To produce flow-up minimum generating sets for splines over $\Z/p^k\Z$, we start with generators for splines over $\Z/p^{\beta}\Z$ for $1\leq \beta \leq k$, pull back to splines over $\Z/p^{\beta+1}\Z$, and add in the missing generators. We give a method for identifying splines over $\Z/p^{\beta}\Z$ with splines over $\Z/p^{\beta+1}\Z$ and give an explicit formula for the additional generators at each step. 

We use the following criteria given by Bowden and Tymoczko \cite{BT} to identify when a generating set is minimum.

\begin{theorem}[Bowden-Tymoczko]\label{minimum gen set crit}
Suppose that $\{\mx{b^{(i_1)}}, \mx{b^{(i_2)}}, \dots, \mx{b^{(i_t)}}\}$ is a set of flow-up generators for the ring of splines $R_{G,\alpha}$ satisfying the following properties:

\begin{itemize}
\item The spline $\mx{b^{(i_1)}} = \mx{b^{(1)}}$ the spline with no leading zeros
\item The splines $\{\mx{b^{(i_s)}} \,\, : \,\, s = 1,2, \dots, t \}$ are constant flow-up splines with $\mx{b^{(i_s)}}_v \in \{0,a_s\}$ for each $v\in V$ and each $s$. 
\item The set $\{ 1=a_1, a_2, \dots, a_t\}$ can be reordered so that $a_{j_1} \, \vline \,a_{j_2}\, \vline \, a_{j_3} \,\vline\, \dots \,\vline \,a_{j_t}$.
\end{itemize}

Then $\{\mx{b^{(i_1)}}, \mx{b^{(i_2)}}, \dots, \mx{b^{(i_t)}}\}$ forms a flow-up minimum generating set for $R_{G,\alpha}$.

\end{theorem}

\noindent Bowden and Tymoczko use $\mx{b^{(1)}}=\mx{1}$ the identity spline in \cite{BT} but the above condition is equivalent. 

One of our main tools is a subset of vertices of an edge-labeled graph called zero-connected components, on which labels in a spline must be equal.  

\begin{definition}
Let $G$ be a connected graph in $\Z/m\Z$. Define an equivalence relation on the vertices of $G$ by $v_i\sim_0 v_j$ if and only if $v_i$ and $v_j$ are connected by a path of edges labeled zero.  We call such a path a zero path, and we call an equivalence class defined by this relation a zero-connected component. 
\end{definition}

\begin{remark}\label{zero cc same}
Recall from Remark \ref{zero_edge_same_label_} that any vertices connected by a zero path must have the same label in a spline. Therefore all vertices in a zero-connected component have the same label in a spline. Proposition \ref{same label} gives a partial converse to this remark.
\end{remark}

In our algorithm we consider the same graph with edge labels in different base rings. In particular we are concerned with the zero-connected components of the same graph with edge labels in $\Z/p^i\Z$ for $i=1,2,\dots, k$. Let $V^{(i,\beta)}$ denote the the zero-connected component of a graph $G$ with edge labels in $\Z/p^\beta\Z$ whose smallest-index vertex is $v_i$.

\begin{example}
Figure \ref{zero cc ex 8}, Figure \ref{zero cc ex 4}, and Figure \ref{zero cc ex 2} show the same graph with edge labels in $\Z/8\Z$, $\Z/4\Z$ and $\Z/2\Z$ respectively. The zero-connected components of the graph in Figure \ref{zero cc ex 8} are  $V^{(1,3)}=\{v_1\}, \, V^{(2,3)}=\{v_2\}, \,V^{(3,3)}=\{v_3\}, \,V^{(4,3)}=\{v_4\}$. The zero-connected components of the graph in Figure \ref{zero cc ex 4} are $V^{(1,2)}=\{v_1\}, \, V^{(2,2)}=\{v_2,v_3,v_4\}$. The only zero-connected component of the graph in Figure \ref{zero cc ex 2} is $V^{(1,1)}=\{v_1,v_2,v_3,v_4\}$. 
\begin{figure}[h!]

\begin{minipage}[c]{0.32\linewidth}
\centering
\begin{tikzpicture}[shorten >=1pt,auto,node distance=1.7cm,
  thick,main node/.style={rectangle,draw,font=\sffamily\bfseries}]
  
 \node[main node] (1) {$v_1$};
 \node[main node] (2) [above right of = 1] {$v_2$};
 \node[main node] (3) [above left of = 2] {$v_3$};
 \node[main node] (4) [below left of =3] {$v_4$};

  \path[every node/.style={color=black, font=\sffamily\small}]
    
    (2) edge node {\color{red}$2$} (1)
    (1) edge node {\color{red}$2$} (4)
    (4) edge node {\color{red}$2$} (3)
    	edge node {\color{red}$4$} (2)
    (3) edge node {\color{red}$4$} (2);

\end{tikzpicture}
\caption{Graph over $\Z/8\Z$}\label{zero cc ex 8}

\end{minipage}
\begin{minipage}[c]{0.32\linewidth}
\centering
\begin{tikzpicture}[shorten >=1pt,auto,node distance=1.7cm,
  thick,main node/.style={rectangle,draw,font=\sffamily\bfseries}]
  
 \node[main node] (1) {$v_1$};
 \node[main node] (2) [above right of = 1] {$v_2$};
 \node[main node] (3) [above left of = 2] {$v_3$};
 \node[main node] (4) [below left of =3] {$v_4$};

  \path[every node/.style={color=black, font=\sffamily\small}]
    
    (2) edge node {\color{red}$2$} (1)
    (1) edge node {\color{red}$2$} (4)
    (4) edge node {\color{red}$2$} (3)
    	edge node {\color{red}$0$} (2)
    (3) edge node {\color{red}$0$} (2);

\end{tikzpicture}

\caption{Graph over $\Z/4\Z$}\label{zero cc ex 4}
\end{minipage}
\begin{minipage}[c]{0.32\linewidth}
\centering
\begin{tikzpicture}[shorten >=1pt,auto,node distance=1.7cm,
  thick,main node/.style={rectangle,draw,font=\sffamily\bfseries}]
  
 \node[main node] (1) {$v_1$};
 \node[main node] (2) [above right of = 1] {$v_2$};
 \node[main node] (3) [above left of = 2] {$v_3$};
 \node[main node] (4) [below left of =3] {$v_4$};

  \path[every node/.style={color=black, font=\sffamily\small}]
    
    (2) edge node {\color{red}$0$} (1)
    (1) edge node {\color{red}$0$} (4)
    (4) edge node {\color{red}$0$} (3)
    	edge node {\color{red}$0$} (2)
    (3) edge node {\color{red}$0$} (2);

\end{tikzpicture}
\caption{Graph over $\Z/2\Z$}\label{zero cc ex 2}
\end{minipage}

\end{figure}

\end{example}

\begin{remark}\label{prime power labels}
Zero-divisors in $\Z/p^k\Z$ are multiples of $p^\beta$ for some $0\leq \beta \leq k$.  If $u$ is a unit modulo $p^k$ the product  $u \cdot p^\beta$ generates the same ideal as $p^\beta$.  Thus the  minimal coset representatives in $\Z$ of the zero-divisors of $\Z/p^k\Z$ are the prime powers $p^\beta$ where $0\leq \beta \leq k$. Following the convention of Remark \ref{edge convention},  the only edge labels in $\Z/p^k\Z$ are $ p^\beta $ where $0\leq \beta \leq k$.
\end{remark}

The following proposition defines a spline that we use in the algorithm to produce minimum generating sets for splines over $\Z/p^k\Z$.

\begin{proposition}\label{new spline}
Let $G$ be a connected graph with edge-labeling $\alpha: E \to \{\text{ideals of } \Z/p^\beta\Z\}$ and let $[\Z/p^\beta\Z]_{G,\alpha}$ be the ring of splines on $(G,\alpha)$. The following vertex-labeling is a spline in $[\Z/p^\beta\Z]_{G,\alpha}$.
\[
\mx{b^{(i,\beta)}}_{v} = 
\begin{cases}
p^{\beta-1} & \text{if  } v \in V^{(i,\beta)} \\
0 & \text{if  } v \not\in V^{(i,\beta)} \\
\end{cases}
\]
Moreover if $\mx{b}$ is a spline that is zero except on $V^{(i,\beta)}$ and there is at least one edge labeled $p^{\beta-1}$ in $(G,\alpha)$ then $\mx{b} = c\mx{b^{(i,\beta)}}$ for some constant $c$ in $\Z/p^\beta\Z$.
\end{proposition}

\begin{proof}
Let $v_j$ and $v_k$ be distinct vertices of the graph $G$. If vertices $v_j$ and $v_k$ are either both in $V^{(i,\beta)}$ or both not in $V^{(i,\beta)}$ then the difference $\mx{b^{(i,\beta)}}_{v_j}-\mx{b^{(i,\beta)}}_{v_k} = 0$ which always satisfies the spline condition. Suppose that $v_j\in V^{(i,\beta)}$ and $v_k\notin V^{(i,\beta)}$ so by construction $\mx{b^{(i,\beta)}}_{v_j}-\mx{b^{(i,\beta)}}_{v_k} = p^{\beta-1}$ is the difference of the vertex labels. The edge labels of $(G,\alpha)$ are $ p^{\gamma} $ for $1\leq \gamma \leq \beta$ or $ 0 $ as discussed in Remark \ref{prime power labels}. Since $v_j $ and $v_k$ are not in the same zero-connected component, their connecting edge is not labeled zero. Therefore $\mx{b^{(i,\beta)}}_{v_j}-\mx{b^{(i,\beta)}}_{v_k} = p^{\beta-1} \in \langle p^{\gamma} \rangle$ satisfying the spline condition. Furthermore, if $\gamma=\beta-1$ over $\Z/p^\beta\Z$ for at least one of the edge labels of $(G,\alpha)$ then to satisfy the spline condition, the difference $\mx{b^{(i,\beta)}}_{v_j}-\mx{b^{(i,\beta)}}_{v_k}$ must be at least $p^{\beta-1}$. Therefore when there is at least one $p^{\beta-1}$ edge label, for any spline $\mx{b}$ that is zero except on $V^{(i,\beta)}$, where it is constant, there exists $c\in \Z/p^\beta\Z$ such that $\mx{b} = c\mx{b^{(i,\beta)}}$. 
\end{proof}

We produce a minimum generating set for splines over $\Z/p\Z$ using the splines $\mx{b^{(i,1)}}$ defined in Proposition \ref{new spline}.

\begin{theorem}\label{mod p algorithm}
Let $G$ be a connected graph and let $[\Z/p\Z] _{G,\alpha}$ be the ring of splines over $G$ with edge-labeling function $\alpha:E \to \{\text{ideals of } \Z/p\Z\}$. Then the set $\mathbb{B}_p=\{ \mx{b^{(i,1)}}\,\, | \,\, V^{(i,1)}\subset V\}$ defined in Proposition \ref{new spline} is a flow-up minimum generating set for $[\Z/p\Z] _{G,\alpha}$. 

\end{theorem}

\begin{proof}
We show that $\mathbb{B}_p$ generates all splines in $[\Z/p\Z] _{G,\alpha}$. Let $\mx{p}$ be an arbitrary spline in $[\Z/p\Z] _{G,\alpha}$. All vertices in a zero-connected component have the same label so $\mx{p}_{v_j} = \mx{p}_{v_i}$ for all $v_j \in V^{(i,1)}$. We write $\mx{p}$ as a linear combination of splines in $\mathbb{B}_p$ as such

\[
\mx{p} = \sum_{V^{(i,1)}} \mx{p}_{v_i}\mx{b^{(i,1)}}
\]

The set $\mathbb{B}_p$ generates $[\Z/p\Z]_{G,\alpha}$ and meets the criteria in Theorem \ref{minimum gen set crit}, so it is a minimum generating set for $[\Z/p\Z]_{G,\alpha}$.
\end{proof}

In the results to follow we denote the standard quotient map $\rho_\beta: \Z/p^\ell\Z \to \Z/p^\beta\Z$ for $\beta$ and $\ell$ with  $1\leq \beta\leq \ell \leq k$. Additionally let $\alpha_\beta$ denote the edge-labeling function that sends each edge $uv$ to the ideal $\alpha_\beta(uv) = \rho_\beta(\alpha(uv))$.

We build upon the formula given in Theorem \ref{mod p algorithm} to algorithmically determine a flow-up minimum generating set for splines in $[\Z/p^k\Z]_{G,\alpha}$. We consider the edge-labeled graph $(G,\alpha_\beta)$ and recursively find minimum generating sets $\mathbb{B}_{p^\beta}$ by, in effect, adding the necessary elements to $\mathbb{B}_{p^{\beta-1}}$. We begin with the following lemma, which gives a natural way to identify a spline $\mx{p} \in [\Z/p^\beta\Z]_{G,\alpha_\beta}$ with a vertex-labeling in $(\Z/p^\ell\Z)^{|V|}$ and shows that the vertex-labeling is a spline in $[\Z/p^\ell\Z]_{G,\alpha_\ell}$.

\begin{lemma}\label{pull back to beta}
Let $G$ be a graph with edge-labeling function $\alpha: E \to \{ \text{ideals of } \Z/p^\ell\Z\}$. Let $\mx{p}$ be a spline in $[\Z/p^{\beta}\Z]_{G,\alpha_\beta}$ where $1\leq \beta\leq \ell$, and let $\mx{\bar{p}}$ be a vertex-labeling in $ \left(\Z/{p^{\ell}\Z}\right)^{|V|}$ so that $\mx{p}_v$ and $\mx{\bar{p}}_v$ have the same minimal coset representative in $\Z$ for each $v\in V$. Then $\mx{\bar{p}}$ is a spline in $[\Z/p^\ell\Z]_{G,\alpha_\ell}$.
\end{lemma}

\begin{proof}
Because the quotient maps $\Z\to \Z/m\Z,\, \Z/m\Z\to \Z/p^k\Z$ and $\Z \to \Z/p^k\Z$ commute when where $p^k$ divides $m$, so do the induced quotient maps on the rings of splines. Thus it is sufficient to show that the the pull back of a spline on $G$ over $\Z/m\Z$ is a spline on $G$ over the integers. So as to easily apply this lemma to a later result (Theorem \ref{integer theorem}), we choose $m = lcm( \ell_1, \ell_2, \dots, \ell_k)\cdot p_1$ where $\ell_1, \ell_2, \dots, \ell_k$ are generators of the edge-label ideals and $p_1$ is the smallest prime factor of $\ell_1 \ell_2 \dots \ell_k$. 

Let $\nu: E \to \{ \text{ideals of } \Z\}$ be an edge-labeling function on $G$ and let $\rho:\Z\to\Z/m\Z$ be the standard quotient map. Finally let $\nu_m$ denote the edge-labeling function that sends each edge $uv$ to the ideal $\nu_m(uv) = \rho_m(\nu(uv))$. Suppose $\mx{q}_v$ is a spline in $[\Z/m\Z]_{G,\nu_m}$. We claim that a vertex-labeling $\mx{p} \in \Z^{|V|}$ with $\rho_m(\mx{p}_v)=\mx{q}_v$ is a spline in $\Z_{G,\nu}$ regardless of the coset representative chosen for $\mx{q}_v$.

 Let $\mx{q}$ be a spline in $[\Z/m\Z]_{G,\nu_m}$ so at each vertex 

\[\mx{q}_v = a + m\Z\]

Let $\mx{p} \in \Z^{|V|}$ be a vertex-labeling with $\rho_m(\mx{p}_v)=\mx{q}_v$. Suppose that $\nu(uv)= \langle \ell_{uv} \rangle$. By definition a spline $\mx{q} \in [\Z/m\Z]_{G,\nu_m}$ is such that 

\[\mx{q}_u - \mx{q}_v = \ell_{uv}t_{uv} \mod m\]

for some constant $t_{uv}$ for all adjacent vertices $u,v$. By construction $\ell_{uv}$ divides $m$, say $m=w\ell_{uv}$. This means 

\[\begin{array}{lll}
\mx{p}_u-\mx{p}_v & = \mx{q}_u - \mx{q}_v \mod m&\\
&&\\
& =\ell_{uv}t_{uv} \mod m &\\
&& \\
& = \ell_{uv}t_{uv} + rm&\\
&& \\
& = \ell_{uv}(t_{uv} + rw)&\in \langle \ell_{uv}\rangle\\
\end{array}\]

for a constant $r$. This holds for all edges $uv$ so $\mx{p}$ is a spline in $\Z_{G,\nu}$. 
\end{proof}

Given a set of splines $\mathbb{B}_{p^{\beta-1}}$ in $[\Z/p^{\beta-1}\Z]_{G,\alpha_{\beta-1}}$ we use $\overline{\mathbb{B}}_{p^{\beta-1}}$ to denote the set of splines in $[\Z/p^\beta\Z]_{G,\alpha_\beta}$ to which $\mathbb{B}_{p^{\beta-1}}$ is identified.

\begin{remark}\label{edge btwn zero cc}
Consider the edge labeled graphs $(G,\alpha_\beta)$ and $(G,\alpha_{\beta-1})$ whose edges are labeled with ideals in $\Z/p^{\beta}\Z$ and $\Z/p^{\beta-1}\Z$ respectively. Suppose there is a vertex $v$ such that $v\in V^{(i,\beta-1)}$ and $v\not\in V^{(i,\beta)}$ for some $\beta$. That is $v$ is in the zero-connected component indexed by $v_i$ in $(G,\alpha_{\beta-1})$ but $v$ is not in the zero-connected component indexed by $v_i$ in $(G,\alpha_{\beta})$. Let $V^{(j,\beta)}$ be the zero-connected component of $(G,\alpha_{\beta})$ that contains $v$. Then there is necessarily at least one edge with label $p^{\beta-1}$ that connects a vertex in $V^{(i,\beta)}$ with a vertex in $V^{(j,\beta)}$.
\end{remark}

The following algorithm uses Theorem \ref{mod p algorithm} as a base case. We show inductively that the flow-up minimum generating set $\mathbb{B}_{p^k}$ is the union of the identification in $\Z/p^k\Z$ of $\mathbb{B}_{p^{k-1}}$ and a particular set of splines of the form defined in Proposition \ref{new spline}.

\begin{algorithm}\label{mod p^k algorithm}

Let $G$ be a graph with edge-labeling function $\alpha: E \to \{ \text{ideals of } \Z/p^k\Z\}$. Let $I_{i}$ denote the indices of the zero-connected components of $(G,\alpha_{i})$. Suppose that $\mathbb{B}_{p^{\beta-1}}$ is a flow-up minimum generating set for $[\Z/p^{\beta-1}\Z]_{G,\alpha_{\beta-1}}$ with each generator associated to a unique zero-connected component indexed by $I_{\beta-1}$. 

We recursively define a set of splines as follows.  

\begin{itemize}
\item Find the flow-up minimum generating set $\mathbb{B}_{p}$ for $[\Z/p \Z] _{G,\alpha_{1}}$ defined in Theorem \ref{mod p algorithm}.

\item Then define

\begin{equation}\label{p^beta mgs}
\mathbb{B}_{p^{\beta}} =\overline{\mathbb{B}}_{p^{\beta-1}} \cup \{ \mx{b^{(i,\beta)}} \,\, | \,\, V^{(i,\beta)} \subseteq (G,\alpha_\beta) \textup{  for which $i \not \in I_{\beta - 1}$}\}
\end{equation}
\end{itemize}
\end{algorithm}

The following theorem states that the set of splines $\mathbb{B}_{p^\beta}$ defined by Algorithm \ref{mod p^k algorithm} is a flow-up minimum generating set for $[\Z/p^\beta\Z]_{G,\alpha_\beta}$. We give an inductive proof of the claim. 

\begin{theorem}\label{mod p^k theorem}
Let $G$ be a graph with edge-labeling function $\alpha: E \to \{ \text{ideals of } \Z/p^k\Z\}$. The set $\mathbb{B}_{p^{\beta}} $ defined by Algorithm \ref{mod p^k algorithm} is a flow-up minimum generating set for $[\Z/p^\beta\Z]_{G,\alpha}$.
\end{theorem}

\begin{proof}
We give a proof by induction that $\mathbb{B}_{p^{\beta}}$ generates $[\Z/p^\beta\Z]_{G,\alpha_\beta}$. The base case that $\mathbb{B}_p$ generates the ring of splines $[\Z/p\Z]_{G,\alpha_1}$ was proven in Theorem \ref{mod p algorithm}. The inductive hypothesis is that the set $\mathbb{B}_{p^{\beta-1}}$ generates $[\Z/p^{\beta-1}\Z]_{G,\alpha_{\beta-1}}$. 

Let $\mx{p}$ be a spline in $[\Z/p^{\beta}\Z]_{G,\alpha_{\beta}}$. Define a spline $\mx{p'}$ as follows
 
 \[
\mx{p'}_v = 
\begin{cases}
\mx{p}_v & \text{  if  }  v \in V^{(i,\beta)} \text{  for  } i \in I_{\beta-1}\\
\mx{p}_u & \text{  if  }  v\in  V^{(i,\beta)} \text{  for  } i \in I_\beta \setminus I_{\beta-1} \\
\end{cases}
\]

where $v,u \in V^{(i,\beta)}$ for $i\in I_{\beta-1}$. That is, $\mx{p'}$ and $\mx{p}$ agree on all vertices $v$ that are in zero-connected components of $(G,\alpha_\beta)$ indexed by elements of $I_{\beta-1}$. And $\mx{p'}$ is constant on each zero-connected component in $(G,\alpha_\beta)$ not indexed by an element in $I_{\beta-1}$.

We first show that $\mx{p'}$ is generated by $\overline{\mathbb{B}}_{p^{\beta-1}}$. Since $\mx{p'}$ is constant on each zero-connected component in $(G,\alpha_\beta)$ that is not indexed by an element in $I_{\beta-1}$,  $\mx{p'}$ is the pull back of a spline in $[\Z/p^{\beta-1}\Z]_{G,\alpha_{\beta-1}}$. Let $\rho^*: [\Z/p^{\beta}\Z]_{G,\alpha_{\beta}} \to [\Z/p^{\beta-1}\Z]_{G,\alpha_{\beta-1}}$ be the induced quotient map on the ring of splines. Then $\rho^*(\mx{p'})$ is generated by $\mathbb{B}_{p^{\beta-1}}$ by our induction hypothesis. Therefore the pull back $\mx{p'}$ is generated by the pull back set $\overline{\mathbb{B}}_{p^{\beta-1}}$.


Let $V^{(j,\beta)}$ be a zero-connected component of $(G,\alpha_{\beta})$ for which $j\in I_\beta \setminus I_{\beta-1}$. Then $\mx{p}$ and $\mx{p'}$ do not (necessarily) agree on $v\in V^{(j,\beta)}$ and at these vertices they differ by a multiple of $p^{\beta-1}$ by Remark \ref{edge btwn zero cc}. The vertices $v\in V^{(j,\beta)}$ for $j\in I_\beta \setminus I_{\beta-1}$ are associated with $\mx{b^{(j,\beta)}}$ defined in Proposition \ref{new spline}. By construction the non-zero entries of $\mx{b^{(j,\beta)}}$ are $p^{\beta-1}$, so we have


\[
\mx{p}= \mx{p'} + \sum_{V^{(j,\beta)}} c_j\mx{b^{(j,\beta)}} 
\]

for the appropriate constants $c_j$, so $\mathbb{B}_{p^\beta}$ generates $[\Z/p^{\beta}\Z]_{G,\alpha_{\beta}}$. 

Finally the set $\mathbb{B}_{p^\beta}$ is a flow-up set of splines by construction and a minimum generating set because it meets the criteria given by Bowden and Tymoczko \cite{BT} in Theorem \ref{minimum gen set crit}.
\end{proof}

\begin{example} In this example we illustrate Algorithm \ref{mod p^k algorithm}. We begin with the graph in Figure \ref{alg ex 8} whose edges are labeled with elements in $\Z/8\Z$. Then we quotient out the edge labels by $4\Z$ and $2\Z$ to get the graphs in Figure \ref{alg ex 4} and Figure \ref{alg ex 2} respectively. 

To construct a flow-up minimum generating set, we first consider the graph in Figure \ref{alg ex 2} whose only zero-connected component is $V^{(1,1)}$ with index set $I_1=1$. The spline $(1\,1\,1\,1)^T$ generates all splines on this graph. Next we consider the graph in Figure \ref{alg ex 4} whose zero-connected components are $V^{(1,2)}$ and $V^{(2,2)}$ with index set $I_2 = 1,2$. Over $\Z/4\Z$, the zero-connected component with index not in $I_1$ is $V^{(2,2)}$ to which we assign the spline $(2\,2\,2\,0)^T$. The pull back of $(1\,1\,1\,1)^T$ with $(2\,2\,2\,0)^T$ generates all splines on the graph in Figure \ref{alg ex 4}. Finally we consider the graph in Figure \ref{alg ex 8} whose zero-connected components are $V^{(1,3)}$, $V^{(2,3)}$, $V^{(3,3)}$ and $V^{(4,3)}$ with index set $I_3=1,2,3,4$. Over $\Z/8\Z$, the zero-connected components with index not in $I_2$ are $V^{(3,3)}$ and $V^{(4,3)}$ to which we assign the splines $(0\,4\,0\,0)^T$ and $(4\,0\,0\,0)^T$ respectively. Together with the pull back of the generators for splines on the graph in Figure \ref{alg ex 4} these generate all splines on the graph in Figure \ref{alg ex 8}. 

\begin{figure}[h!]

\begin{minipage}[c]{0.45\linewidth}
\centering
\begin{tikzpicture}[shorten >=1pt,auto,node distance=1.7cm,
  thick,main node/.style={rectangle,draw,font=\sffamily\bfseries}]
  
 \node[main node] (1) {$v_1$};
 \node[main node] (2) [above right of = 1] {$v_2$};
 \node[main node] (3) [above left of = 2] {$v_3$};
 \node[main node] (4) [below left of =3] {$v_4$};

  \path[every node/.style={color=black, font=\sffamily\small}]
    
    (2) edge node {\color{red}$2$} (1)
    (1) edge node {\color{red}$2$} (4)
    (4) edge node {\color{red}$2$} (3)
    	edge node {\color{red}$4$} (2)
    (3) edge node {\color{red}$4$} (2);

\end{tikzpicture}
\caption{Graph over $\Z/8\Z$}\label{alg ex 8}

\end{minipage}
\quad\quad
\begin{minipage}[c]{0.3\linewidth}
\centering

$\left\{ \begin{pmatrix} 4\\0\\0\\0 \end{pmatrix},\begin{pmatrix} 0\\4\\0\\0 \end{pmatrix}, \begin{pmatrix} 2\\2\\2\\0 \end{pmatrix}, \begin{pmatrix} 1\\1\\1\\1 \end{pmatrix} \right\}$

\end{minipage}
\vspace{5mm}

\begin{minipage}[c]{0.45\linewidth}
\centering
\begin{tikzpicture}[shorten >=1pt,auto,node distance=1.7cm,
  thick,main node/.style={rectangle,draw,font=\sffamily\bfseries}]
  
 \node[main node] (1) {$v_1$};
 \node[main node] (2) [above right of = 1] {$v_2$};
 \node[main node] (3) [above left of = 2] {$v_3$};
 \node[main node] (4) [below left of =3] {$v_4$};

  \path[every node/.style={color=black, font=\sffamily\small}]
    
    (2) edge node {\color{red}$2$} (1)
    (1) edge node {\color{red}$2$} (4)
    (4) edge node {\color{red}$2$} (3)
    	edge node {\color{red}$0$} (2)
    (3) edge node {\color{red}$0$} (2);

\end{tikzpicture}

\caption{Graph over $\Z/4\Z$}\label{alg ex 4}
\end{minipage}
\quad\quad
\begin{minipage}[c]{0.3\linewidth}
\centering

$\left\{ \begin{pmatrix} 2\\2\\2\\0 \end{pmatrix}, \begin{pmatrix} 1\\1\\1\\1 \end{pmatrix} \right\}$

\end{minipage}
\vspace{5mm}

\begin{minipage}[c]{0.45\linewidth}
\centering
\begin{tikzpicture}[shorten >=1pt,auto,node distance=1.7cm,
  thick,main node/.style={rectangle,draw,font=\sffamily\bfseries}]
  
 \node[main node] (1) {$v_1$};
 \node[main node] (2) [above right of = 1] {$v_2$};
 \node[main node] (3) [above left of = 2] {$v_3$};
 \node[main node] (4) [below left of =3] {$v_4$};

  \path[every node/.style={color=black, font=\sffamily\small}]
    
    (2) edge node {\color{red}$0$} (1)
    (1) edge node {\color{red}$0$} (4)
    (4) edge node {\color{red}$0$} (3)
    	edge node {\color{red}$0$} (2)
    (3) edge node {\color{red}$0$} (2);

\end{tikzpicture}
\caption{Graph over $\Z/2\Z$}\label{alg ex 2}
\end{minipage}
\quad\quad
\begin{minipage}[c]{0.3\linewidth}
\centering

$\left\{ \begin{pmatrix} 1\\1\\1\\1 \end{pmatrix} \right\}$

\end{minipage}

\end{figure}

\end{example}

\subsection{Minimum generating sets mod $m$}\label{subsection mod m algorithm}

To produce minimum generating sets for splines over $\Z/m\Z$, we combine our algorithm in Section \ref{mod p^k sect} with a structure theorem for splines over $\Z/m\Z$ given by Bowden and Tymoczko \cite{BT}. 

Bowden and Tymoczko \cite{BT} give an algorithm to produce a minimum generating set for splines over $\Z/m\Z$ from the minimum generating sets for splines over $\Z/p_i^{e_i}\Z$ $i=1,2,\dots,t$ where $p_1^{e_1}\dots p_t^{e_t}$ is the primary decomposition of $m$. Specifically Bowden and Tymoczko prove the following rings of splines are isomorphic.

\begin{theorem}[Bowden-Tymoczko]\label{primary decomp}
Suppose that $m=p_1^{e_1}p_2^{e_2}\cdots p_t^{e_t}$ is the primary decomposition of $m$. For each $i=1,2, \dots, t$ let $\rho_{e_i} : \Z/m\Z \to \Z/p_i^{e_i}\Z$. Let $G$ be a graph with edge-labeling $\alpha$ over the integers modulo $m$, and for each $e_i$ let $\alpha_{e_i}$ denote the edge-labeling function that sends each edge $uv$ to the ideal $\alpha_{e_i}(uv) = \rho_{e_i}(\alpha(uv))$. Then 

\[
[\Z/m\Z]_{G,\alpha} \cong \bigoplus_{i=1}^t [\Z/p_i^{e_i}\Z]_{G,\alpha_{e_i}}
\]

via the induced map $\rho^*:[\Z/m\Z]_{G,\alpha} \to \bigoplus_{i=1}^t [\Z/p_i^{e_i}\Z]_{G,\alpha_{e_i}}$ defined by $(\rho^*(\mx{p}))_v = \bigoplus_{i=1}^t \rho_{e_i}(\mx{p}_v)$.
\end{theorem}

The uniqueness of the minimum generating set modulo $m$ pulled back from the minimum generating sets modulo the primary decomposition of $m$ is given by the Chinese Remainder Theorem.

\begin{algorithm}\label{mod m algorithm}
Let $G$ be a graph with edge-labeling function $\alpha:E \to \{\text{ideals of } \Z/m\Z\}$. Suppose that $m=p_1^{e_1}p_2^{e_2}\cdots p_t^{e_t}$ is the primary decomposition of $m$. For each $i=1,2, \dots, t$ let $\rho_{e_i} : \Z/m\Z \to \Z/p_i^{e_i}\Z$ and for each $i$ let $\alpha_{e_i}$ denote the edge-labeling function that sends the edge $uv$ to the ideal $\alpha_{e_i}(uv) = \rho_{e_i}(\alpha(uv))$. 

We define a set of splines as follows

\begin{enumerate}
\item Use Algorithm \ref{mod p^k algorithm} to determine the minimum generating set $\mathbb{B}_{p_i^{e_i}}$ of the ring of splines $[\Z/p_i^{e_i}\Z]_{G,\alpha_i}$ for each $i=1,2,\dots,t$.

\item\label{project splines} Select a spline $\mx{b}$ from each $\mathbb{B}_{p_i^{e_i}}$ and find the unique spline in $[\Z/m\Z]_{G,\alpha}$ that projects onto $\mx{b} \in \mathbb{B}_{p_i^{e_i}}$ for each $i=1,2,\dots, t$. Use the zero spline if all others in some $\mathbb{B}_{p_i^{e_i}}$ have previously been selected.

\item Repeat Step \ref{project splines} using each $\mx{b} \in \mathbb{B}_{p_i^{e_i}}$ only once for each $i=1,2,\dots,t$ and using $\mx{0}\in \mathbb{B}_{p_i^{e_i}}$ when necessary. 
\end{enumerate}

Let $\mathbb{B}_m$ denote the set of splines in $[\Z/m\Z]_{G,\alpha}$ produced by the above steps.
\end{algorithm}

\begin{theorem}
Let $G$ be a graph with edge-labeling function $\alpha:E \to \{\text{ideals of } \Z/m\Z\}$. Suppose that $m=p_1^{e_1}p_2^{e_2}\cdots p_t^{e_t}$ is the primary decomposition of $m$. For each $i=1,2, \dots, t$ let $\rho_{e_i} : \Z/m\Z \to \Z/p_i^{e_i}\Z$ and for each $i$ let $\alpha_{e_i}$ denote the edge-labeling function that sends the edge $uv$ to the ideal $\alpha_{e_i}(uv) = \rho_{e_i}(\alpha(uv))$. 

\begin{enumerate}
\item \label{claim 1} The set of splines $\mathbb{B}_m$ produced by Algorithm \ref{mod m algorithm} is a minimum generating set for $[\Z/m\Z]_{G,\alpha}$ when arbitrary splines $\mx{b} \in \mathbb{B}_{p_i^{e_i}}$ for each $i=1,2,\dots,t$ are selected in Step \ref{project splines} of Algorithm \ref{mod m algorithm}. 

\item\label{claim 2} The set of splines $\mathbb{B}_m$ produced by Algorithm \ref{mod m algorithm} is a flow-up minimum generating set for $[\Z/m\Z]_{G,\alpha}$ when splines $\mx{b} \in \mathbb{B}_{p_i^{e_i}}$ are selected as follows in Step \ref{project splines} of Algorithm \ref{mod m algorithm} 

\begin{itemize} 
\item\label{same leading zeros} Order the elements of $\mathbb{B}_{p_i^{e_i}}$ according to number of leading zeros from fewest to most for each $i=1,\dots,t$. Select a spline $\mx{b}$ from each $\mathbb{B}_{p_i^{e_i}}$ in this order.
%

\end{itemize}
\end{enumerate}
\end{theorem}

\begin{proof}

The set $\mathbb{B}_m$ generates $[\Z/m\Z]_{G,\alpha}$ by Theorem \ref{primary decomp}. Each set $\mathbb{B}_{p_i^{e_i}}$ is minimum by Theorem \ref{mod p^k theorem} so the generating set $\mathbb{B}_m$ is also minimum. This proves Claim \ref{claim 1}.

To prove Claim \ref{claim 2} we need to check that the pull back to the $[\Z/m\Z]_{G,\alpha}$ is a flow-up set of generators. Let $\mx{f_r}$ be the unique spline in $[\Z/m\Z]_{G,\alpha}$ that projects onto the $r^{th}$ basis vector (or zero vector) in each ordered set $\mathbb{B}_{p_i^{e_i}}$ of generators.  Each $\mx{f_r}$ selected in Claim \ref{claim 2} is a flow-up spline with leading entry at some vertex, say $v_r$. By construction Claim \ref{claim 2} produces only one spline $\mx{f_r} \in[\Z/m\Z]_{G,\alpha}$ with leading entry at $v_r$. Therefore the set $\mathbb{B}_m$ in Claim \ref{claim 2} is a flow-up minimum generating set.
\end{proof}

\subsection{An algorithm to produce a basis over the integers}

To conclude this section we extend Algorithm \ref{mod m algorithm} to produce a basis for a ring of splines $\Z_{G,\alpha}$. We carefully pick the modulus $m$ and utilize Algorithm \ref{mod m algorithm} over $\Z/m\Z$. This produces a basis rather than just a minimum generating set because the $\Z$-module of splines $\Z_{G,\alpha}$ is free. Before presenting the algorithm we discuss the motivation for and prove the necessary facts about each step.
  
A pivotal step of the algorithm is the selection of a particular modulus $m$ for which the ring of splines over $\Z/m\Z$ pulls back nicely to splines over $\Z$. To this end we use Lemma \ref{pull back to beta} in the proof of Theorem \ref{integer theorem}. 

In order to get a basis for splines over the integers, we use the following characterization of flow-up bases for splines over the integers due to Bowden, Hagen, King, and Reinders \cite{BowdenHKR}.

\begin{theorem}[Bowden-Hagen-King-Reinders] \label{min leading elt}
Let $G$ be a graph on $|V|=n$ vertices with edge-labeling function $\alpha: E \to \{ \text{ideals of } \Z\}$. The following are equivalent:

\begin{itemize}
\item The set $\{ \mx{b^{(1)}}, \dots, \mx{b^{(n)}} \}$ forms a flow-up basis for splines on $\Z_{G,\alpha}$.

\item For each flow-up spline $\mx{f^{(i)}}=(0, \dots, \mx{f^{(i)}}_{v_i} , \dots, \mx{f^{(i)}}_{v_n})^T$ in $\Z_{G,\alpha}$ the entry $\mx{f^{(i)}}_{v_i}$ is an integer multiple of the entry $\mx{b^{(i)}}_{v_i}$.
\end{itemize}
\end{theorem}

First we construct a minimal leading non-zero element for splines over integers modulo $m$. 

\begin{lemma}\label{minimize leading}
Let $G$ be a connected graph and let $[\Z/m\Z]_{G,\alpha_m}$ be the ring of splines over $G$ with edge-labeling function $\alpha_m: E \to \{ \text{ ideals of } \Z/m\Z \}$. Let $\{ \mx{b^{(1)}}, \dots, \mx{b^{(s)}} \}$ be the flow-up minimum generating set for $[\Z/m\Z]_{G,\alpha_m}$. Define the spline $\mx{{f}^{(i)}}  \in [\Z/m\Z]_{G,\alpha_m}$  by
\[
\mx{{f}^{(i)}} = x\cdot \mx{b^{(i)}} \mod m
\]

where $x$ is an element in $\Z/m\Z$ such that $x\cdot \mx{b^{(i)}}_{v_i}=gcd\left(\mx{b^{(i)}}_{v_i}\, ,\, m \right)$. Then $\mx{{f}^{(i)}}_{v_i}$ has smallest minimal coset representative in $\Z$ of all elements in $\langle \mx{b^{(i)}} \rangle$.
\end{lemma}

\begin{proof}
This is essentially a result about subgroups of cyclic groups. Recall that for a finite cyclic group of order $m$, there is exactly one subgroup of order $m_i$ for each divisor $m_i$ of $m$. The smallest generator of the subgroup is $\langle d \rangle$  is the greatest common divisor of $m$ and $d$. Applying this to the cyclic subgroup of $\Z/m\Z$ generated by the spline entry $\mx{b^{(i)}}_{v_i}$ we conclude the minimal non-zero leading entry $v_i$ of a splines in $\langle \mx{b^{(i)}} \rangle$ is $gcd\left(\mx{b^{(i)}}_{v_i}\, ,\, m \right)$. Multiplying $\mx{b^{(i)}}$ by $x\in \Z/m\Z$ for which $x\cdot \mx{b^{(i)}}=gcd\left(\mx{b^{(i)}}_{v_i}\, ,\, m \right)$ yields a spline $\mx{{f}^{(i)}}$ whose leading element has smallest minimal coset representative in $\Z$ of all splines in $\langle \mx{b^{(i)}} \rangle$. 
\end{proof}

The following algorithm to produces a set of integer splines uses Algorithm \ref{mod p^k algorithm} and its extension to $\Z/m\Z$. Theorem \ref{integer theorem} employs the information given in Lemma \ref{pull back to beta} and Lemma \ref{minimize leading} to prove that this algorithm produces a basis for splines over the integers. 

\begin{algorithm}\label{integer algorithm}
Let $G$ be a graph with on $|V|=n$ vertices $\alpha: E \to \{ \text{ ideals of } \Z\}$. Let $\rho_m: \Z \to \Z/m\Z$ be the standard quotient map and let $\alpha_m$ denote the edge-labeling function that sends each edge $uv$ to the ideal $\alpha_m(uv) = \rho_m(\alpha(uv))$. 

\begin{enumerate}
\item Fix the modulus $m = lcm( \ell_1, \ell_2, \dots, \ell_t)\cdot p_1$ where $\ell_1, \ell_2, \dots, \ell_t$ are the smallest generators of the edge-label ideals $\alpha(uv)$ and $p_1$ is the smallest factor of the product $\ell_1 \ell_2 \dots \ell_t$. 

\item Use Algorithm \ref{mod m algorithm} to determine $\mathbb{B}_m$ a flow-up minimum generating set of $[\Z/m\Z]_{G,\alpha_m}$.

\item \label{minimize spline} For each $i=1,2,\dots,n$ and $\mx{b^{(i)}}$ in $\mathbb{B}_m$ compute the spline $\mx{{f}^{(i)}}=x\cdot \mx{b^{(i)}} \mod m$ where $x$ is the element in $\Z/m\Z$ such that $x\cdot \mx{b^{(i)}}_{v_i}=gcd\left(\mx{b^{(i)}}_{v_i}\, ,\, m \right)$.

\item Identify each spline $\mx{{f}^{(i)}} \in [\Z/m\Z]_{G,\alpha_m}$ with any integer spline $\mx{\bar{f}^{(i)}}$ such that $\mx{\bar{f}^{(i)}}_{v_t} = a$ where $a$ is the smallest coset representative in $\Z$ of $\mx{f^{(i)}}_{v_t}$ for each vertex $v_t$.
\end{enumerate}

Let $\mathbb{B}=\{ \mx{\bar{f}^{(1)}},  \mx{\bar{f}^{(2)}}, \dots ,  \mx{\bar{f}^{(n)}}\}$ denote the set of splines produced by this algorithm. 

\end{algorithm}

\begin{theorem}\label{integer theorem}
Let $G$ be a graph on $|V|=n$ vertices with $\alpha: E \to \{ \text{ ideals of } \Z\}$. Let $\rho_m: \Z \to \Z/m\Z$ be the standard quotient map and let $\alpha_m$ denote the edge-labeling function that sends each edge $uv$ to the ideal $\alpha_m(uv) = \rho_m(\alpha(uv))$. Let $\rho^*_m: \Z_{G,\alpha} \to [\Z/m\Z]_{G,\alpha_m}$ be the induced quotient map on the module of splines. The splines $\mathbb{B}=\{ \mx{\bar{f}^{(1)}},  \mx{\bar{f}^{(2)}}, \dots ,  \mx{\bar{f}^{(n)}}\}$ defined in Algorithm \ref{integer algorithm} form a flow-up basis for $\Z_{G,\alpha}$.
\end{theorem}

\begin{proof}
To prove that $\mathbb{B}$ is a flow-up basis for $\Z_{G,\alpha}$ we show that there is a flow-up generator for each of $n$ distinct vertices with each generator satisfying the criterion given by Bowden, Hagen, King, and Reinders in Theorem \ref{min leading elt}. 

Gilbert, Polster, and Tymoczko \cite{GilbertPolsterTymoczko} showed that the minimum number of generators for splines over $\Z$ is equal to the number of distinct vertices of $G$. Suppose the modulus $m = lcm( \ell_1, \ell_2, \dots, \ell_t)\cdot p_1$ has primary decomposition $m=p_1^{e_1}, \dots p_r^{e_r}$. If we take edge labels modulo $p_1^{e_1}$ then $G$ has no zero-labeled edges and thus has $n$ zero-connected components. We show in Theorem \ref{rank theorem} in Section \ref{rank sect} that this means Algorithm \ref{mod m algorithm} produces a set of flow-up splines with cardinality $n$. By Lemma \ref{pull back to beta} any preimage in $\Z^{n}$ of a spline in $[\Z/m\Z]_{G,\alpha_m}$ is also a spline in $\Z_{G,\alpha}$ for our particular selection of $m$. Therefore the splines $\{ \mx{\bar{f}^{(1)}},  \mx{\bar{f}^{(2)}}, \dots ,  \mx{\bar{f}^{(n)}}\}$ are splines in $\Z_{G,\alpha}$. 

We show by contradiction that each spline in $\{ \mx{\bar{f}^{(1)}},  \mx{\bar{f}^{(2)}}, \dots ,  \mx{\bar{f}^{(n)}}\}$ has a minimal leading element in $\Z_{G,\alpha}$. Suppose for the sake of contradiction that there is some flow-up spline $\mx{p^{(i)}}$ at $v_i$ in $\Z_{G,\alpha}$ for which $\mx{p^{(i)}}_{v_i} < \mx{\bar{f}^{(i)}}_{v_i}$. By construction we have the following images

\[\rho_m(\mx{p^{(i)}}_{v_i})=a +m\Z \quad \text{  and  } \quad \rho_m(\mx{\bar{f}^{(i)}}_{v_i}) = \mx{f^{(i)}}_{v_i}=b+m\Z\]
where $a<b$. Moreover we know that $\mx{p^{(i)}} \in \langle \mx{b^{(i)}} \rangle$ where $\mx{b^{(i)}}\in\mathbb{B}_m$ the flow-up minimum generating set produced by Algorithm \ref{mod m algorithm}. We showed in Lemma \ref{minimize leading} that $\mx{f^{(i)}}$ has smallest minimal coset representative in $\Z$ of all elements in $\langle \mx{b^{(i)}} \rangle$ so we have reached a contradiction. Therefore $\mx{\bar{f}^{(i)}}$ is a spline in $\Z_{G,\alpha}$ with minimal leading element, so the set $\mathbb{B}$ meets the criterion for a basis given by Bowden, Hagen, King, and Reinder's in Theorem \ref{min leading elt}. This concludes the proof that $\mathbb{B}$ is a basis for $\Z_{G,\alpha}$.
\end{proof}

\begin{remark}
If in application we want to minimize each of entry of a spline in the flow-up basis $\mathbb{B}$, we can achieve this recursively. First we complete Algorithm \ref{integer algorithm} then minimize $\mx{\bar{f}^{(n-1)}}$ by subtracting the largest multiple of $\mx{\bar{f}^{(n)}}$ such that all entries of their difference are non-negative. We continue this process, minimizing $\mx{\bar{f}^{(i-1)}}$ by subtracting the largest multiple of $\mx{\bar{f}^{(i)}}$ from $\mx{\bar{f}^{(i-1)}}$ such that all entries of their difference are non-negative.
\end{remark}

\section{Minimum Number of Generators for Splines mod $m$}\label{rank sect}

In this section we determine the minimum number of generators for splines over $\Z/m\Z$ on a particular graph. Bowden and Tymoczko \cite{BT} showed that depending on the modulus $m$, the minimum number of generators of a ring of splines over $\Z/m\Z$ can vary between $1$ and $|V|$. We sharpen their result to give a quick graph-theoretic test of the minimum number of generators for splines over $\Z/m\Z$ using the work in Section \ref{section algorithm mod m}. 

Using Algorithm \ref{mod p^k algorithm} and the isomorphism in Theorem \ref{primary decomp} we are able to classify the condition under which two vertices have identical labels in all splines over $\Z/m\Z$ on a particular graph. This is a partial converse of Remark \ref{zero cc same}.  

\begin{proposition}\label{same label}
Let $G$ be an edge-labeled graph with edge-labeling function $\alpha: E \to \{ \text{ideals of } \Z/m\Z\}$. Let $m=p_1^{e_1}p_2^{e_2}\dots p_t^{e_t}$ be the primary factorization of $m$. For each $i=1,2, \dots, t$ let $\rho_{e_i} : \Z/m\Z \to \Z/p_i^{e_i}\Z$, and for each $e_i$ let $\alpha_{e_i}$ denote the edge-labeling function that sends each edge $uv$ to the ideal $\alpha_{e_i}(uv) = \rho_{e_i}(\alpha(uv))$. For any spline $\mx{f}$ in $[\Z/m\Z]_{G,\alpha}$, the vertices $v_1$ and $v_2$ must have the same labels $\mx{f}_{v_1} = \mx{f}_{v_2}$ if and only if $v_1$ and $v_2$ are in the same zero-connected component of  $(G,\alpha_{e_i})$ for each $i=1,2,\dots,t$. 
\end{proposition}

\begin{proof}
First assume $v_1$ and $v_2$ are in the same zero-connected component in $(G,\alpha_{e_i})$ for each $i=1,2,\dots,t$. By Remark \ref{zero cc same} each spline agrees on $v_1$ and $v_2$ for each $i=1,2,\dots,t$.  By the Chinese Remainder Theorem each spline modulo $m$ must also agree on $v_1$ and $v_2$.

For the other direction we give a proof of the contrapositive.  Assume $v_1$ and $v_2$ are in different zero-connected components in $(G,\alpha_{e_j})$ for some $j$ with $1\leq j\leq t$.  Let $V^{(s,e_j)}$ be the zero-connected component containing $v_1$ in $(G,\alpha_{e_j})$, and let $\mx{b}^{(s,e_j)}$ be the spline defined in Proposition \ref{new spline}. By construction, $\mx{b}^{(s,e_j)}$ differs on $v_1$ and $v_2$. Let $\mx{f}$ be the pull back in $[\Z/m\Z]^{|V|}$ of zero splines in $[\Z/p_i^{e_i}\Z]_{G,\alpha_{e_i}}$ for $i\neq j$ and $\mx{b}^{(s,e_j)}$ in $[\Z/p_j^{e_j}\Z]_{G,\alpha_{e_j}}$. Bowden and Tymoczko's isomorphism \cite{BT} shows that $\mx{f}$ is a spline in $[\Z/m\Z]_{G,\alpha}$ which by construction differs on $v_1$ and $v_2$. This proves the claim.
\end{proof}

The following theorem has two parts. The first part gives the minimum number of generators for a ring of splines over $\Z/p^k\Z$. The second part gives the minimum number of generators for a ring of splines over $\Z/m\Z$. The proof  follows from the construction of minimum generating sets in Section \ref{section algorithm mod m}.

\begin{theorem}\label{rank theorem}
Let $G$ be graph with edge-labeling function $\alpha: E \to \{\text{ ideals of } \Z/m\Z\,\,\}$. Suppose that $m=p_1^{e_1}p_2^{e_2}\cdots p_t^{e_t}$ is the primary decomposition of $m$. For each $i=1,2, \dots, t$ let $\rho_{e_i} : \Z/m\Z \to \Z/p_i^{e_i}\Z$ be the standard quotient map, and for each $i$ let $\alpha_{e_i}$ denote the edge-labeling function that sends the edge $uv$ to the ideal $\alpha_{e_i}(uv) = \rho_{e_i}(\alpha(uv))$. Then 

\begin{enumerate}


\item\label{rank p^k} The minimum number of generators for $[\Z/p_i^{e_i}\Z]_{G,\alpha_i}$ is equal to the number of zero-connected components of $(G,\alpha_{e_i})$. 

\item\label{rank m} The minimum number of generators for $[\Z/m\Z]_{G,\alpha}$ is equal to the maximum of the minimum number of generators for $[\Z/p_i^{e_i}\Z]_{G,\alpha_{e_i}}$ over all $i=1,2,\dots, t$.
\end{enumerate}
\end{theorem}

\begin{proof}
Claim \ref{rank p^k} follows from the construction of $\mathbb{B}_{p_i^{e_i}}$ in Algorithm \ref{mod p^k algorithm}. By construction, each zero-connected component of $(G,\alpha_{e_i})$ indexes exactly one spline in $\mathbb{B}_{p_i^{e_i}}$. The set $\mathbb{B}_{p_i^{e_i}}$ is minimum by Theorem \ref{mod p^k theorem} so the minimum number of generators for $[\Z/p_i^{e_i}\Z]_{G,\alpha_{e_i}}$ is equal to the number of zero-connected components of $(G,\alpha_{e_i})$.

Claim \ref{rank m} follows from Claim \ref{rank p^k} and Bowden and Tymoczko's Theorem \ref{primary decomp}. Every spline in each $\mathbb{B}_{p_i^{e_i}}$ is pulled back to $[\Z/m\Z]_{G,\alpha}$, so the maximum of the minimum number of generators for $[\Z/p_i^{e_i}\Z]_{G,\alpha_i}$ over all $i=1,2,\dots, t$ is equal to the minimum number of generators for splines over $[\Z/m\Z]_{G,\alpha}$.
\end{proof}

In practice we generally start with graphs whose edges are labeled by non-zero ideals, in which case this theorem says that a ring of splines over $\Z/p^k\Z$ has $|V|$ generators.

\section{Multiplication Tables}

In this section we discuss multiplication tables of the minimum generating sets for splines over connected graphs produced by the algorithm determined in Section \ref{section algorithm mod m}. We are motivated by topological constructions of cohomology rings as splines, where we want explicit formulas for multiplication of ring elements. We find explicit formulas for multiplication tables over $\Z/m\Z$ when $m$ is either the product of distinct primes or just a prime power. We do not have explicit formulas for a general $m$ but can use Theorem \ref{primary decomp} to identify an algorithm for multiplication tables.

Multiplication of splines is defined entry-wise. For instance the product of two splines $\mx{f}$ and $\mx{p}$ is defined for each entry $v_i$ as

\[\mx{g}_{v_i} =  \mx{f}_{v_i}\cdot\mx{p}_{v_i}\]

The entry-wise multiplication of splines is computed in the base ring $R$. 

\subsection{Multiplication of splines modulo a product of distinct primes}

We now give multiplication tables for splines over the integers modulo a product of distinct primes.  

\begin{definition}
The support of a spline is the collection of vertices with non-zero entries. We denote the support of a spline $\mathbf{f}$ as supp$(\mathbf{f})$.
\end{definition}

The next theorem treats multiplication of splines in the minimum generating set given by Algorithm \ref{mod m algorithm}. We show that each of these splines is idempotent (and thus also a zero divisor) in the ring of splines over $\Z/m\Z$. 

\begin{theorem}\label{distinct primes}
Let $G$ be a graph with edge-labeling function $\alpha: E \to \{\text{ideals in  } \Z/p_1\dots p_s\Z\}$ where $\{p_1,p_2,\dots, p_s\}$ are distinct primes. Let $\mathbb{B}_{p_1\dots p_s}$ be the minimum generating set of $[\Z/p_1\dots p_s\Z]_{G,\alpha}$ from Algorithm \ref{mod m algorithm}. Then the multiplication table of $[\Z/p_1\dots p_s\Z]_{G,\alpha}$ is

\[\mx{b^{(i)}}\mx{b^{(j)}}=
\begin{cases}
\mx{0} & \text{if  } i\neq j\\
\mx{b^{(i)}} & \text{if  } i= j\\
\end{cases}
\]
\end{theorem}

\begin{proof}
We prove the claim by induction. Recall the image modulo $p_r$ for $r=1,\dots, s$ of each $\mx{b^{(i)}}$ is determined by a unique zero-connected component as given by Proposition \ref{new spline}. 

Our base case is for splines over $\Z/p_1\Z$. For $i\neq j$, since $\mx{b^{(i)}}$ and $\mx{b^{(j)}}$ are each non-zero exactly on a distinct zero-connected component, if $\mx{b^{(i)}}_{v_t}\equiv 1 \mod p$ then $\mx{b^{(j)}}_{v_t} \equiv 0 \mod p$ and vice versa. Therefore $\mx{b^{(i)}}_{v_t}\mx{b^{(j)}}_{v_t} \equiv 0 \mod p$ for all vertices $v_t$ so $\mx{b^{(i)}}\mx{b^{(j)}} = \mx{0}$ by definition. The non-zero entries of $ \mx{b^{(i)}}$ are $1$ by construction. Therefore $\mx{b^{(i)}}_{v_t}\mx{b^{(i)}}_{v_t} = \mx{b^{(i)}}_{v_t}$ for all vertices $v_t$ so $\mx{b^{(i)}}\mx{b^{(i)}} = \mx{b^{(i)}}$.

Assume as our induction hypothesis that for all $k\leq s-1$ the minimum generating set $\mathbb{B}_{p_1\dots p_{s-1}}$ has the multiplication table given above. For each $\mx{b^{(i)}}$ in $\mathbb{B}_{p_1, \dots, p_s}$, Algorithm \ref{mod m algorithm} gives us the following for each vertex $v_t$ in the graph.

 \[\mx{b^{(i)}}_{v_t} \equiv \mx{c^{(i)}}_{v_t} \mod p_1\dots p_{s-1} \quad\quad\text{and}\quad\quad \mx{b^{(i)}}_{v_t} \equiv \mx{d^{(i)}}_{v_t} \mod p_s\]

where $\mx{c^{(i)}}$ and $\mx{d^{(i)}}$ are elements of $\mathbb{B}_{p_1, \dots, p_{s-1}}$ and $\mathbb{B}_{p_s}$ respectively. We first show that $\mx{b^{(i)}}\mx{b^{(j)}} = \mx{0}$ when $i\neq j$. For each vertex our base case and induction hypothesis give  

\[\mx{c^{(i)}}_{v_t}\mx{c^{(j)}}_{v_t}=0   \mod p_1\dots p_{s-1}
\quad\quad\text{and}\quad\quad 
 \mx{d^{(i)}}_{v_t}\mx{d^{(j)}}_{v_t} =0 \mod p_{s}
\quad\quad\text{when}\quad i\neq j\]

We conclude that   

\[\mx{b^{(i)}}_{v_t}\mx{b^{(j)}}_{v_t}=
\begin{cases}
0 & \mod p_1\dots p_{s-1}\\
0 & \mod p_s\\
\end{cases}
\]

so by the Chinese Remainder Theorem $\mx{b^{(i)}}_{v_t}\mx{b^{(j)}}_{v_t} \equiv 0 \mod p_1\dots p_s$ for each vertex $v_t$. Therefore all elements $\mx{b^{(i)}} \in\mathbb{B}_{p_1 \dots p_s}$ have multiplicative structure $\mx{b^{(i)}}\mx{b^{(j)}} = \mx{0}$ for $i\neq j$. 

Now we show for each $\mx{b^{(i)}} \in \mathbb{B}_{p_1\dots p_s}$ we have $\mx{b^{(i)}}\mx{b^{(i)}} = \mx{b^{(i)}}$.  In this case the induction hypothesis and base case show 

\[\mx{c^{(i)}}_{v_t}\mx{c^{(i)}}_{v_t}= \mx{c^{(i)}}_{v_t}   \mod p_1\dots p_{s-1}
\quad\quad\text{and}\quad\quad 
 \mx{d^{(i)}}_{v_t}\mx{d^{(i)}}_{v_t} =\mx{d^{(i)}}_{v_t} \mod p_{s}
\quad\quad\text{when}\quad i\neq j\]

We conclude that  

\[\mx{b^{(i)}}_{v_t}\mx{b^{(i)}}_{v_t}=
\begin{cases}
\mx{c^{(i)}}_{v_t} & \mod p_1\dots p_{s-1}\\
\mx{d^{(i)}}_{v_t} & \mod p_s\\
\end{cases}
\]
which again by the Chinese Remainder Theorem gives  $\mx{b^{(i)}}_{v_t}\mx{b^{(i)}}_{v_t} \equiv \mx{b^{(i)}}_{v_t} \mod p_1\dots p_s$ for each vertex $v_t$. Therefore  $\mx{b^{(i)}}\mx{b^{(i)}} = \mx{b^{(i)}}$ for each $\mx{b^{(i)}} \in \mathbb{B}_{p_1\dots p_s}$. Equivalently each of these elements is idempotent in $[\Z/p_1\dots p_s\Z]_{G,\alpha}$. 
\end{proof}

\begin{example}\label{distinct prime example}
Consider the three cycle over $\Z/30\Z$ shown in Figure \ref{mult pqr example}. 

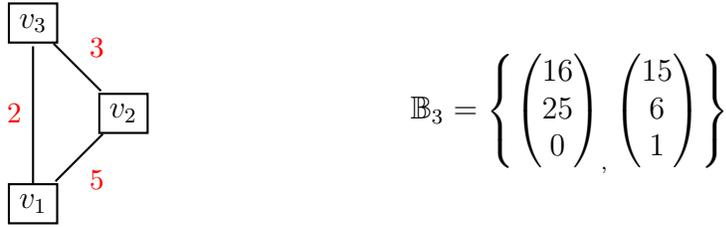
\begin{figure}[h!]
  \begin{minipage}[c]{0.3\textwidth}
  \begin{tikzpicture}[shorten >=1pt,auto,node distance=1.7cm,
  thick,main node/.style={rectangle,draw,font=\sffamily\bfseries}]
  
  \node[main node] (3) {$v_3$};
  \node[main node] (2) [below right of = 3] {$v_2$};
  \node[main node] (1) [below left of = 2] {$v_1$};

  \path[every node/.style={color=black, font=\sffamily\small}]
    
    (3) edge node {\color{red}$3$} (2)
    (1) edge node {\color{red}$2$} (3)
    (2) edge node {\color{red}$5$} (1);

\end{tikzpicture}
\end{minipage}
\quad
\begin{minipage}[c]{0.3\textwidth}
  
$
\mathbb{B}_{3} = \left \{ \begin{pmatrix} 16 \\ 25 \\ 0 \end{pmatrix}_, 
\begin{pmatrix} 15 \\ 6 \\ 1  \end{pmatrix} \right \}$
\end{minipage}
\caption{Edge-Labeled Graph and minimum generating set in $\Z/30\Z$}\label{mult pqr example}
\end{figure}
\vspace{5mm}
The reader can confirm that the spline product $(15\,\, 6\,\, 1)^T(16 \,\,25 \,\,0)^T = \mx{0}$ and that both splines are idempotent modulo $30$. 
\end{example}

\subsection{Multiplication modulo a prime power}

We now build multiplication tables for splines over $\Z/p^k\Z$. The key observation is that Algorithm \ref{mod p^k algorithm} produces minimum generating sets whose elements have either disjoint or nested supports. We prove this in the following lemma.

\begin{lemma}\label{support lemma}
Let $G$ be a graph with edge-labeling function $\alpha: E \to \{\text{ideals of } \Z/p^k\Z\}$ and let $\alpha_i$ denote the edge-labeling function that sends the edge $uv$ to the ideal $\alpha_\beta(uv) = \rho_\beta(\alpha(uv))$ where $\rho_\beta: \Z/p^k\Z \to \Z/p^\beta\Z$ is the standard quotient map. The supports of splines in the minimum generating set $\mathbb{B}_{p^k}$ are either disjoint or one is a proper subset of the other. Moreover suppose the supports of $\mx{b^{(i)}}, \, \mx{b^{(j)}}$ determined by Algorithm \ref{mod p^k algorithm} are not disjoint and that $\mx{b^{(i)}}$ has non-zero entries $p^s$ while $\mx{b^{(j)}}$ has non-zero entries $p^r$ with $s<r$. Then the support of $\mx{b^{(j)}}$ is a proper subset of the support of $\mx{b^{(i)}}$. 
\end{lemma}

\begin{proof}
The support of each spline in $\mathbb{B}_{p^k}$ is a zero-connected component of $(G,\alpha_\beta)$ for some $1\leq \beta \leq k$ by construction. Each zero-connected component of $(G,\alpha_\beta)$ is contained within a zero-connected component of $(G,\alpha_{\beta-1})$ for all $\beta > 1$. Further, zero-connected components of $(G,\alpha_\beta)$ for each $\beta$ are disjoint by definition. Therefore the supports of two splines in $\mathbb{B}_{p^k}$ are either disjoint or nested.

Suppose the supports of $\mx{b^{(i)}}, \, \mx{b^{(j)}}$ in $\mathbb{B}_{p^k}$ are not disjoint and that $\mx{b^{(i)}}$ has non-zero entries $p^s$ while $\mx{b^{(j)}}$ has non-zero entries $p^r$ with $s<r$. If the non-zero entries of the $\mx{b^{(i)}}$ are $p^s$ then the support of $\mx{b^{(i)}}$ is the zero-connected component $V^{(i,s+1)}$ by construction in Algorithm \ref{mod p^k algorithm}. Similarly if the non-zero entries of $\mx{b^{(j)}}$ are $p^r$ then the support of $\mx{b^{(j)}}$ is the zero-connected component $V^{(j,r+1)}$. Since by assumption the zero-connected components $V^{(i,s+1)}$ and $V^{(j,r+1)}$ are not disjoint, they must be nested.  Since $s<r$ we conclude $V^{(j,r+1)} \subseteq V^{(i,s+1)}$. Therefore the support of $\mx{b^{(j)}}$ is a proper subset of the support of $\mx{b^{(i)}}$.

\end{proof}

The following theorem presents the multiplication table for splines in the minimum generating set $\mathbb{B}_{p^k}$. We use the characteristics described in Lemma \ref{support lemma} to prove this result. 

\begin{theorem}\label{prime power mult}
Let $G$ be a graph with edge-labeling function $\alpha: E \to \{\text{ideals of } \Z/p^k\Z\}$ and let $\mx{b^{(i)}}$ and $\mx{b^{(j)}}$ be splines in the minimum generating set $\mathbb{B}_{p^k}$ produced by the Algorithm \ref{mod p^k algorithm}. Suppose also that the non-zero entries of $\mx{b^{(i)}}$ are $p^s$ and those of $\mx{b^{(j)}}$ are $p^r$ where $s\leq r<k$. Then the multiplication table of  $[\Z/p^k\Z]_{G,\alpha}$ is

\[
\mx{b^{(i)}}\mx{b^{(j)}} =
\begin{cases}
p^s\mx{b^{(j)}}&  \text{if  }\quad \text{supp}(\mx{b^{(i)}})\bigcap \text{supp}(\mx{b^{(j)}})\neq \emptyset  \text{ and } s+r < k\\ 
\mx{0} & \text{if  } \quad \text{supp}(\mx{b^{(i)}})\bigcap \text{supp}(\mx{b^{(j)}})\neq \emptyset  \text{ and } s+r \geq k\\
\mx{0} & \text{if  } \quad \text{supp}(\mx{b^{(i)}})\bigcap \text{supp}(\mx{b^{(j)}})=\emptyset \\
\end{cases}
\] 
\end{theorem}

\begin{proof}
If the supports of $\mx{b^{(i)}}$ and $\mx{b^{(j)}}$ are disjoint then entry-wise multiplication of $\mx{b^{(i)}}$ and $\mx{b^{(j)}}$ gives $\mx{b^{(i)}}\mx{b^{(j)}} = \mx{0}$. 

Suppose the supports of $\mx{b^{(i)}}$ and $\mx{b^{(j)}}$ have nonempty intersection. For any vertex $v_t$ at which both $\mx{b^{(i)}}$ and $\mx{b^{(j)}}$ are non-zero we have 

\[ \mx{b^{(i)}}_{v_t}\mx{b^{(j)}}_{v_t} = p^sp^r = p^{s+r}\]

If $s+r \geq k$ this is identically zero so $\mx{b^{(i)}}\mx{b^{(j)}}=\mx{0}$. If $s+r<k$ then by Lemma \ref{support lemma} we know $\text{supp}(\mx{b^{(j)}})$ is a proper subset of $\text{supp}(\mx{b^{(i)}})$, so $\text{supp}(\mx{b^{(j)}}) = \text{supp}(\mx{b^{(i)}}\mx{b^{(j)}})$. Therefore $\mx{b^{(i)}}\mx{b^{(j)}} = p^s\mx{b^{(j)}}$ if $s+r < k$.
\end{proof} 

\begin{example}
Consider the graph over $\Z/8\Z$ shown in Figure \ref{mult p^k example}. 

\begin{figure}[h!]
\centering
  \begin{minipage}[c]{0.3\textwidth}
  \begin{tikzpicture}[shorten >=1pt,auto,node distance=1.7cm,
  thick,main node/.style={rectangle,draw,font=\sffamily\bfseries}]
  
  \node[main node] (4) {$v_4$};
  \node[main node] (3) [right of = 4] {$v_3$};
  \node[main node] (2) [below of = 3] {$v_2$};
  \node[main node] (1) [below of = 4] {$v_1$};

  \path[every node/.style={color=black, font=\sffamily\small}]
    
    (4) edge node {\color{red}$2^2$} (3)
    (1) edge node {\color{red}$2$} (4)
    (3) edge node {\color{red}$2^2$} (2)
    	edge node {\color{red}$2$} (1)
    (2) edge node {\color{red}$2$} (1);

\end{tikzpicture}\end{minipage}
\quad
\centering
\begin{minipage}[c]{0.45\textwidth}
  
$
\mathbb{B}_{\Z/3\Z} = \left \{ \begin{pmatrix} 1 \\ 1 \\ 1\\1 \end{pmatrix}_, 
\begin{pmatrix} 2 \\ 2 \\ 2 \\ 0  \end{pmatrix} _,
\begin{pmatrix} 0 \\ 2^2 \\ 0 \\ 0  \end{pmatrix}_,
\begin{pmatrix} 2^2 \\ 0 \\ 0 \\ 0  \end{pmatrix}\right \}$
\end{minipage}
\caption{Edge-Labeled Graph and minimum generating set in $\Z/8\Z$} \label{mult p^k example}
\end{figure}
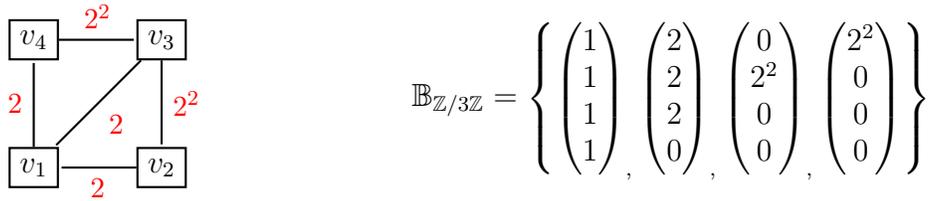
\vspace{5mm}

Multiplying two splines with disjoint supports like $( 0 \,\,0 \,\,2^2\,\, 0)^T$ and $( 0\,\, 0 \,\,0 \,\,2^2 )^T$ yields the zero spline. This also occurs when the supports are not disjoint but the powers of $2$ in the non-zero entries sum to at least $3$ as with $( 0\,\, 0 \,\,,2^2 \,\,0)^T$ and $( 0 \,\,2 \,\,2 \,\,2 )^T$.  When the powers of $2$ in the non-zero entries sum to less than $3$ we get a scalar multiple of one spline. For instance 
\[( 0 \,\,2 \,\,2 \,\,2 )^T( 0 \,\,2 \,\,2 \,\,2 )^T = 2( 0 \,\,2 \,\,2 \,\,2 )^T\]
\end{example}

\subsection{Multiplication modulo $m$}

We do not have an explicit formula for multiplication tables of splines over $\Z/m\Z$ for a general $m$ as we did in Theorem \ref{distinct primes} for splines over $\Z/p_1\dots p_s\Z$ and in Theorem \ref{prime power mult} for splines over $\Z/p^k\Z$. However we can algorithmically determine multiplication tables.
\begin{remark}
Let $G$ be graph with edge-labeling function $\alpha: E \to \{\text{ ideals of } \Z/m\Z\,\,\}$. Suppose that $m=p_1^{e_1}p_2^{e_2}\cdots p_t^{e_t}$ is the primary decomposition of $m$. For each $i=1,2, \dots, t$ let $\rho_{e_i} : \Z/m\Z \to \Z/p_i^{e_i}\Z$ and let $\alpha_{e_i}$ denote the edge-labeling function that sends the edge $uv$ to the ideal $\alpha_{e_i}(uv) = \rho_{e_i}(\alpha(uv))$. The multiplication table of splines in the minimum generating set $\mathbb{B}_m$ produced by Algorithm \ref{mod m algorithm} is determined by the multiplication tables of splines in the minimum generating sets $\mathbb{B}_{p_i^{e_i}}$ produced by Algorithm \ref{mod p^k algorithm}.
\end{remark}


\section{Acknowledgements}
The authors gratefully acknowledge the guidance and support of Julianna Tymoczko as well as useful conversations with Nealy Bowden, Lauren Olson, Megan Perry, and the support of the Smith College Center for Women in Mathematics and NSF Grant DMS--1143716.

\end{document}